\documentclass[11pt]{amsart}
\usepackage[margin=1in]{geometry}
\usepackage{yufei}
\numberwithin{equation}{section}
\title[The Heim and Neuhauser Conjecture]{Towards Heim and Neuhauser's Unimodality Conjecture on the Nekrasov-Okounkov polynomials}
\date{August 2020}
\author{Letong Hong, Shengtong Zhang}
\address{Department of Mathematics, Massachusetts Institute of Technology, Cambridge, MA 02139}
\email{clhong@mit.edu, stzh1555@mit.edu}

\begin{document}
\begin{abstract}
\vspace{-1em}
Let $Q_n(z)$ be the polynomials associated with the Nekrasov-Okounkov formula
 $$\sum_{n\geq 1} Q_n(z) q^n := \prod_{m = 1}^\infty (1 - q^m)^{-z - 1}.$$
In this paper we partially answer a conjecture of Heim and Neuhauser, which asks if $Q_n(z)$ is unimodal, or stronger, log-concave for all $n \geq 1$. Through a new recursive formula, we show that if $A_{n,k}$ is the coefficient of $z^k$ in $Q_n(z)$, then $A_{n,k}$ is log-concave in $k$ for $k \ll n^{1/6}/\log n$ and monotonically decreasing for $k \gg \sqrt{n}\log n$. We also propose a conjecture that can potentially close the gap.
\vspace{-0.5em}
\end{abstract}
\maketitle
\section{Introduction}
In their groundbreaking work \cite{NO06},  Nekrasov and Okounkov showed the hook length formula\footnote{This formula was also obtained concurrently by Westbury(see Proposition 6.1 and 6.2 of \cite{Westbury06}.) and Han(see \cite{Han})}
\begin{equation}
 \sum_{\lambda}q^{\abs{\lambda}}\prod_{h\in \cH(\lambda)}\bigp{1 + \frac{z}{h^2}} = \prod_{m = 1}^\infty (1 - q^m)^{-z - 1},
\end{equation}
where $\lambda$ runs over all Young tableaux, $\abs{\lambda}$ denotes the size of $\lambda$, and $\cH(\lambda)$ denotes the multiset of hook lengths associated to $\lambda$. We define
\begin{equation}
Q_n(z) = \sum_{\abs{\lambda} = n}\prod_{h\in \cH(\lambda)}\bigp{1 + \frac{z}{h^2}}.
\end{equation}
For example, we can calculate that
\begin{align*}
    Q_0(z) &= 1, \\
    Q_1(z) &= 1 + z, \\
    Q_2(z) &= 2 + \frac{5}{2}z + \frac{1}{2}z^2, \\
    Q_3(z) &= 3 + \frac{29}{6} z + 2z^2 + \frac{1}{6}z^3.
\end{align*}
The polynomials $Q_n(z)$ are of degree $n$ with positive coefficients and satisfy
\begin{equation}
\label{eq:defn}
 \sum_{n = 0}^\infty Q_n(z) q^n = \prod_{m = 1}^\infty (1 - q^m)^{-z - 1}.
\end{equation}
 The study of $Q_n(z)$ was initiated by D'Arcais in \cite{Arcais13}\footnote{D'Arcais defined the polynomial $P_n(z) = Q_n(z - 1)$ via the infinite product, not the hook number expression.}. More recently in \cite{HN18} and \cite{HN20}, Heim and Neuhauser have been investigating the number theoretic and distributional properties of the $Q_n(z).$ They proved the identity (\cite{HN18}, Conjecture 1)
\begin{equation}
\label{eq:Heim}
Q_n(z) =  \sum_{\abs{\lambda} = n}\prod_{h\in \cH(\lambda)^\diamond}\bigp{1 + \frac{z}{h}}
\end{equation}
where $\cH(\lambda)^\diamond$ denotes the multiset of hook lengths associated to $\lambda$ with trivial legs. 

In \cite{HN18}, Heim and Neuhauser conjectured that the polynomials $Q_n(z)$ are \emph{unimodal}. In other words, let the coefficient of $z^k$ in $Q_n(z)$ as $A_{n,k}$; then there exists some integer $k_1\in [0,n]$ such that $A_{n, i} \le A_{n,i + 1}$ when $0 \leq i < k_1$ and $A_{n, i} \ge A_{n,i + 1}$ when $k_1 \leq i < n$. 
They verified via computation that up to $n \leq 1000$, the polynomials $Q_n(z)$ are in fact \emph{log-concave}, which means that $A_{n, k}^2 \geq A_{n, k - 1}A_{n, k + 1}$ for all $1 \leq k \leq n - 1$. In this paper we make partial progress towards Heim and Neuhauser's conjecture. We show that the polynomial $Q_n(z)$ is log-concave at the start, and monotone decreasing at the tail. Throughout the rest of the paper, the constants in $O$, $\gg$ and $\ll$ are absolute unless otherwise stated.
\begin{theorem}
\label{thm:main-1}
For $n$ sufficiently large, we have

(1) For $k \ll \frac{n^{1/6}}{\log n}$, we have $A_{n, k}^2 \geq A_{n, k - 1}A_{n, k + 1}.$

(2) For $k \gg \sqrt{n}\log n$, we have $A_{n, k} \geq A_{n, k + 1}.$
\end{theorem}
\begin{remark}
By ``For $k \ll \frac{n^{1/6}}{\log n}$, we have $\cdots$", we mean that there exists an absolute constant $\kappa > 0$ such that the statement holds for $k < \kappa\frac{n^{1/6}}{\log n}$. We will use this notation throughout the rest of the paper. All the constants could be explicitly computed if one carefully traces the proof.
\end{remark}
We also reduce Heim and Neuhauser's conjecture to a more ``explicit'' conjecture. For positive integers $n$, define
$$\sigma_{-1}(n) = \sum_{d \vert n} d^{-1},$$
and define $f(q)$ to be its generating function
\begin{equation}
f(q) = \sum_{n \geq 1} \sigma_{-1}(n)q^n.    
\end{equation}
We are interested in the behavior of $c_{n,k}$, the coefficient of $q^n$ in $f^k(q)$.\footnote{Using the notation $c_{n,k}$ instead of $c_{k,n}$ is more convenient for our purposes.}
\begin{conjecture}
\label{conjecture}
There exists a constant $C > 1$ such that for all $k \geq 2$ and $n \leq C^k$, we have
$$c^2_{n, k} \geq c_{n - 1, k}c_{n + 1, k}.$$
\end{conjecture}
\begin{remark}
In the last section we offer some numerical computations with regard to this conjecture. We believe that $C=2$ is a viable value in the Conjecture.
\end{remark}
We show that \cref{conjecture} implies Heim and Neuhauser's conjecture for $n$ sufficiently large.
\begin{theorem}
\label{thm:main-2}
If \cref{conjecture} is true, then for all $n \gg \log^{-7} C + 1$, the polynomial $Q_n(z)$ is unimodal. 
\end{theorem}
\begin{remark}
We believe that \cref{conjecture} may imply that for all sufficiently large $n$, the polynomial $Q_n(z)$ is log-concave.
\end{remark}
\section*{Acknowledgements}
We would like to thank Professor Ken Ono for the proposal of the project and his guidance. We thank Professor Bernhard Heim and Professor Markus Neuhauser for carefully reviewing the transcript and providing helpful feedbacks. We also thank Jonas Iskander for many helpful discussions. The research was supported by the generosity of the National Science Foundation under grant DMS-2002265, the National Security Agency under grant H98230-20-1-0012, the Templeton World Charity Foundation, and the Thomas Jefferson Fund at the University of Virginia.

\section{Proof of Theorem \ref{thm:main-1}(1)}
\label{sec:2}
In this section we show Theorem \ref{thm:main-1}(1), which establishes the log-concavity of $A_{n,k}$ when $k \ll n^{1/6} / \log n$. Our proof is organized as follows. In \cref{lemma:recursive}, we establish a key recursive formula for $A_{n,k}$ which relates it to $c_{n,k}$. This allows us to translate log-concavity into an asymptotic estimate for $A_{n,k}$. In \cref{lemma:c-asymp}, we prove an explicit estimate for $c_{n,k}$. This estimate allows us to show, in \cref{lemma:a-asymp}, that $A_{n,k}$ is close to a log-concave sequence in $n$, which is enough to show the desired log-concavity for $k$.
\subsection{A recursive formula for $A_{n,k}$}
Our proof is centered around the following observation.
\begin{lemma}
\label{lemma:recursive}
For any non-negative integers $a < b$, and any $n \geq b$, we have
$$A_{n,b} = \frac{a!}{b!}\sum_{i = 0}^n A_{n - i, a}c_{i, b - a}.$$
\end{lemma}
\begin{proof}
We first note that $f(q)$ is equal to the log derivative of
$$\prod_{m = 1}^\infty (1 - q^m)^{-1}.$$
Let $k$ be any non-negative integer. In \cref{eq:defn}, taking derivative with respect to $z$ for $k$ times, then setting $z = 0$, we get\footnote{Throughout this paper, we define $A_{n,k}$ or $c_{n,k}$ to be $0$ for all undefined subscripts.}
\begin{equation}
    \sum_{n = 0}^\infty A_{n,k} q^n = \frac{1}{k!}f^k(q) \prod_{m = 1}^\infty (1 - q^m)^{-1}.
\end{equation}
Applying the above for $k = a$ and $k = b$, we obtain
\begin{equation}
    \sum_{n = 0}^\infty A_{n,b} q^n = \frac{a!}{b!}f^{b - a}(q) \sum_{n = 0}^\infty A_{n,a} q^n.
\end{equation}
The lemma then follows from the definition of $c_{n, k}$.
\end{proof}
\subsection{An asymptotic for $c_{n,k}$ and $A_{n,k}$}
We now apply Lemma \ref{lemma:recursive} on $(a,b) = (0, k)$, giving 
\begin{equation}
\label{eq:recursive-2}
    A_{n,k} = \frac{1}{k!}\sum_{i = 0}^n A_{n - i, 0}c_{i,k}.
\end{equation}
We observe that $A_{n - i, 0 } = p(n - i)$, where $p(n)$ is the partition function, which satisfies a well-known asymptotic obtained by Hardy and Ramanujan that we shall use below. Thus, to understand the behavior of $A_{n, k}$, it suffices to estimate $c_{i, k}$. However, we are only able to obtain a very crude estimate.
\begin{lemma}
\label{lemma:c-asymp}
For any positive integers $n \geq 2, k$ with $n \geq k^2$, we have
$$\sum_{m \leq n}c_{m, k} = \bigp{\frac{\pi^2}{6}}^k\binom{n}{k}\bigp{1 + O\bigp{\frac{k^2 \log n}{n}}}.$$
\end{lemma}
\begin{proof}
We first note that
\begin{align}
\label{eq:step-1}
\begin{split}
  \sum_{m \leq n}c_{m, k} &= \sum_{a_1 + \cdots + a_k \leq n} \sigma_{-1}(a_1)\cdots \sigma_{-1}(a_k)  \\
  &= \sum_{a_1 + \cdots + a_k \leq n} \sum_{d_1 | a_1, d_2 | a_2,\cdots, d_k | a_k}\frac{1}{d_1d_2\cdots d_k}\\
  &= \sum_{d_1,\cdots, d_k = 1}^n \frac{1}{d_1d_2\cdots d_k}\#\{(x_1,\cdots,x_k)\in \mathbb{Z}_{>0}^k: d_1x_1 + \cdots d_kx_k \leq n\}.\\
 \end{split}
\end{align}
We now show that 
\begin{equation}
\label{eq:counting-points}
\abs{\#\{(x_1,\cdots,x_k)\in \mathbb{Z}_{>0}^k: d_1x_1 + \cdots d_kx_k \leq n\} - \frac{n^k}{k! d_1\cdots d_k}} \leq \frac{n^{k - 1}(d_1 + \cdots d_k)}{(k - 1)! d_1\cdots d_k}.
\end{equation}
For each $(x_1,\cdots,x_k)\in \mathbb{Z}_{>0}^k$ that satisfies $d_1x_1 + \cdots d_kx_k \leq n$, we place a unit cube with the uppermost vertex at $(x_1,\cdots,x_k)$. The union of the cubes are contained in the region 
\begin{equation*}
S_U = \{(x_1,\cdots,x_k)\in \mathbb{R}_+^k: d_1x_1 + \cdots d_kx_k \leq n\}    
\end{equation*}
and contain the region
\begin{equation*}
S_L = \{(x_1,\cdots,x_k)\in \mathbb{R}_+^k: d_1x_1 + \cdots d_kx_k \leq n - d_1 - \cdots - d_k\}.
\end{equation*}
If we denote $V$ as volume, then we have
$$V(S_U) = \frac{n^k}{k!d_1\cdots d_k},$$
and
$$V(S_L) = \frac{(n - d_1 - \cdots - d_k)^k}{k!d_1\cdots d_k}\geq \frac{n^k - k(d_1 + \cdots + d_k)n^{k - 1}}{k!d_1\cdots d_k}$$ by Bernoulli inequality.
Thus we get \eqref{eq:counting-points}.

Plugging \eqref{eq:counting-points} to \eqref{eq:step-1}, we conclude that
\begin{align*}
    \sum_{m \leq n} c_{m, k} &= R + \sum_{d_1,\cdots, d_k = 1}^n \frac{1}{d_1d_2\cdots d_k}\frac{n^k}{k! d_1\cdots d_k} \\
    &= R + \bigp{\frac{\pi^2}{6}}^k\binom{n}{k}\bigp{1 + O\bigp{\frac{k^2}{n}}}
\end{align*}
where the error term $R$ is controlled by
\begin{align*}
    \abs{R} &\leq \sum_{d_1,\cdots, d_k = 1}^n \frac{1}{d_1d_2\cdots d_k}\frac{n^{k - 1}(d_1 + \cdots d_k)}{(k - 1)! d_1\cdots d_k} \\ 
    &= k \frac{n^{k - 1}}{(k - 1)!} \sum_{d_1,\cdots, d_k = 1}^n \frac{1}{d_1d_2^2d_3^2\cdots d_k^2} \\
    &\ll \frac{n^{k - 1}}{(k - 1)!} \cdot k\bigp{\frac{\pi^2}{6}}^{k - 1}\log n.
\end{align*}
Combining the error terms, we conclude the lemma.
\end{proof}
With this lemma we can get an asymptotic estimate for $A_{n,k}$ when $n$ is much larger than $k$.
\begin{lemma}
\label{lemma:a-asymp}
For any positive integers $n \geq 2, k$ with $n \geq k^4 \log^4 k$, we have
\begin{align*}
A_{n,k}= \bigp{1 + O\bigp{(2\log n)^{- k} + \frac{k^2 \log^2 n}{\sqrt{n}}}}\frac{1}{k!} \bigp{\frac{\pi^2}{6}}^k\sum_{i = 0}^n p(n - i) \binom{i - 1}{k - 1}.
\end{align*}
\end{lemma}
\begin{proof}
We recall the Hardy-Ramanujan formula for $p(n)$ in \cite{HR18},
\begin{equation}
    \label{eq:partition-asymptotic}
    p(n) \sim \frac{e^{\pi\sqrt{2/3}\sqrt{n}}}{4\sqrt{3} n}.
\end{equation}
By partial summing the recursive formula \eqref{eq:recursive-2}, we obtain
\begin{align*}
A_{n,k}= \frac{1}{k!} \sum_{i = 0}^n (p(n - i) - p(n - i - 1))\sum_{j \leq i}c_{j, k}
\end{align*}
where we let $p(-1) = 0$.\footnote{Note that $p(n)$ is monotone increasing, thus all the terms are positive.} Notice that the terms with small $i$ could be upper bounded, as we have
$$\sum_{i \leq \sqrt{n} / \log n} (p(n - i) - p(n - i - 1))\sum_{j \leq i}c_{j, k} \leq p(n) \sum_{j \leq \sqrt{n} / \log n}c_{j, k},$$
while we also have
$$\sum_{i \in [2\sqrt{n}, 3\sqrt{n}]} (p(n - i) - p(n - i - 1))\sum_{j \leq i}c_{j, k} \geq (p(n - \ceil{2\sqrt{n}}) - p(n - \floor{3\sqrt{n}})) \sum_{j \leq 2\sqrt{n}}c_{j, k}.$$
By the asymptotic formula for the partition number \eqref{eq:partition-asymptotic}, we obtain
$$p(n) \asymp p(n - \ceil{2\sqrt{n}}) - p(n - \floor{3\sqrt{n}}),$$
while by Lemma \ref{lemma:c-asymp}, we have, for $m = \floor{\sqrt{n} / \log n}$ and $m =  \floor{2\sqrt{n}}$, 
$$\sum_{j \leq m}c_{j, k} = \bigp{1 + O\bigp{\frac{k^2 \log m}{m}}}\bigp{\frac{\pi^2}{6}}^k\binom{m}{k}.$$
When $n \geq k^4\log^4 k$, the big-O term is small, and we get
$$\sum_{j \leq \sqrt{n} / \log n}c_{j, k} \ll (2\log n)^{- k}\sum_{j \leq 2\sqrt{n}}c_{j, k}.$$
Therefore, the terms with small $i$ is small compared with the terms $i \in [2\sqrt{n}, 3\sqrt{n}]$, and we obtain
\begin{align*}
A_{n,k}= \bigp{1 + O((2\log n)^{-k})}\frac{1}{k!} \sum_{i \in [\sqrt{n} / \log n, n]} (p(n - i) - p(n - i - 1))\sum_{j \leq i}c_{j, k}.
\end{align*}
We apply Lemma \ref{lemma:c-asymp}, and obtain
$$\sum_{j \leq i}c_{j, k} = \bigp{1 + O\bigp{\frac{k^2 \log i}{i}}} \bigp{\frac{\pi^2}{6}}^k\binom{i}{k}.$$
Since in the summation we have $i \geq \sqrt{n} / \log n$, we conclude that 
\begin{align*}
A_{n,k}= \bigp{1 + O\bigp{(2\log n)^{- k} + \frac{k^2 \log^2 n}{\sqrt{n}}}}\frac{1}{k!} \sum_{i \in [\sqrt{n} / \log n, n]} (p(n - i) - p(n - i - 1)) \bigp{\frac{\pi^2}{6}}^k\binom{i}{k}.
\end{align*}
We could go through the above argument again to sum starting from $1$ instead of $\sqrt{n} / \log n$ with a negligible error (dominated again by the terms $i \in [2\sqrt{n}, 3\sqrt{n}]$). Simplifying, we get the desired conclusion
\begin{align*}
A_{n,k}= \bigp{1 + O\bigp{(2\log n)^{- k} + \frac{k^2 \log^2 n}{\sqrt{n}}}}\frac{1}{k!} \bigp{\frac{\pi^2}{6}}^k\sum_{i = 1}^n p(n - i) \binom{i - 1}{k - 1}. \qed\popQED
\end{align*}
\end{proof}
\begin{remark}
From this lemma it is easy to derive an explicit asymptotic for $A_{n,k}$ as $n\to\infty$ by simply plugging in the asymptotic for $p(n)$. However, for our application it is simpler to leave $A_{n,k}$ unsimplified in the current form.
\end{remark}
\subsection{Proof of Theorem \ref{thm:main-1}(1)}
We now note that Lemma \ref{lemma:a-asymp} essentially tells us that $A_{n,k}$ is close to a log-concave sequence in $n$. This, together with an application of \cref{lemma:recursive}, directly implies the desired result.
\begin{lemma}
\label{lemma:pseudo-log-concave}
For $n \geq k^4\log^4 k$, we have
$$A_{n,k} = \bigp{1 + O\bigp{(2\log n)^{- k} + \frac{k^2 \log^2 n}{\sqrt{n}}}}\Tilde{A}_{n,k},$$
where $\Tilde{A}_{n,k}$ is a log-concave sequence in $n$.
\end{lemma}
\begin{proof}
Lemma \ref{lemma:a-asymp} says that 
\begin{align*}
A_{n,k}= \bigp{1 + O\bigp{(2\log n)^{- k} + \frac{k^2 \log^2 n}{\sqrt{n}}}}\frac{1}{k!} \bigp{\frac{\pi^2}{6}}^k\sum_{i = 1}^n p(n - i) \binom{i - 1}{k - 1}.    
\end{align*}
By \cref{eq:partition-asymptotic}, we have $p(\ceil{\sqrt{n}}) \gg e^{n^{1/5}}p(25)$, while $\binom{n - 1}{k - 1} \asymp \binom{n - \ceil{\sqrt{n}} - 1}{k - 1}$. Thus the summation term with $i = n - \ceil{\sqrt{n}}$ dominate the tail terms $i \in [n - 25, n]$, and it follows that
\begin{align*}
A_{n,k}= \bigp{1 + O\bigp{(2\log n)^{- k} + \frac{k^2 \log^2 n}{\sqrt{n}}}}\frac{1}{k!} \bigp{\frac{\pi^2}{6}}^k\sum_{i = 1}^{n - 26} p(n - i) \binom{i - 1}{k - 1}.    
\end{align*}
By \cite{DP14}, the sequence $p(n)_{(n \geq 25)}$ is log-concave, and the binomial polynomial $\binom{n}{k - 1}$ is log-concave in $n$. Therefore, the series
$$\widetilde{A}_{n, k} = \frac{1}{k!} \bigp{\frac{\pi^2}{6}}^k\sum_{i = 1}^{n - 26} p(n - i) \binom{i - 1}{k - 1} $$
is the convolution of two log-concave series, and is therefore log-concave by Hoggar's Theorem(\cite{Hoggar}). We thus obtain the lemma.
\end{proof}
\begin{remark}
Numerical evidence suggests that for all $k \geq 2$, the sequence $A_{n,k}$ is log-concave in $n$. Unfortunately, this statement seems to be even harder than Heim and Neuhauser's conjecture.
\end{remark}
\begin{proof}[Proof of Theorem \ref{thm:main-1}(1)]
For convenience, we replace $k$ with $k + 1$. We use Lemma \ref{lemma:recursive} for $(a,b) = (k, k + 1)$ and $(a,b) = (k, k + 2)$, and get
\begin{equation}
\label{eq:final-step-1}
A_{n,k + 1} = \frac{1}{k}\sum_{i = 0}^n A_{n - i, k}\sigma_{-1}(i)    
\end{equation}
and
\begin{equation}
\label{eq:final-step-2}
A_{n,k + 2} = \frac{1}{k(k + 1)}\sum_{i,j \geq 0, i + j \leq n} A_{n - i - j, k}\sigma_{-1}(i)\sigma_{-1}(j).
\end{equation}
By \eqref{eq:partition-asymptotic}, for any $0 \leq x \leq n/2$, we have
$$p(\floor{n/2} + x) \gg e^{0.5\sqrt{n}}p(x).$$
Thus, comparing \cref{eq:recursive-2} term by term, for all $i \leq n / 2$, we have
$$A_{n - 1, k} \gg e^{0.5\sqrt{n}} A_{i, k}.$$
Since $\sigma_{-1}(i) \in [1, 2 + \log i]$, we conclude that the terms in \eqref{eq:final-step-1} with $i \geq \frac{n}{2}$ are all negligible compared to the term $i = 1$. Absorbing them into the error term, we get
$$A_{n,k + 1} = \bigp{1 + O\bigp{e^{-0.4\sqrt{n}}}}\frac{1}{k}\sum_{i = 0}^{\floor{n/2}} A_{n - i, k}\sigma_{-1}(i).$$
Applying Lemma \ref{lemma:pseudo-log-concave}, we get
$$A_{n,k + 1} = \bigp{1 + O\bigp{(2\log n / 2)^{- k} + \frac{k^2 \log^2 n}{\sqrt{n}}}}\frac{1}{k}\sum_{i = 0}^{\floor{n/2}} \Tilde{A}_{n - i, k}\sigma_{-1}(i).$$
Similarly, from \eqref{eq:final-step-2} and Lemma \ref{lemma:pseudo-log-concave} we get
$$A_{n,k} = \bigp{1 + O\bigp{(2\log n / 2)^{- k} + \frac{k^2 \log^2 n}{\sqrt{n}}}}\Tilde{A}_{n, k},$$
and
$$A_{n,k + 2} = \bigp{1 + O\bigp{(2\log n / 2)^{- k} + \frac{k^2 \log^2 n}{\sqrt{n}}}}\frac{1}{k(k + 1)}\sum_{i, j = 0}^{\floor{n/2}} \Tilde{A}_{n - i - j, k}\sigma_{-1}(i)\sigma_{-1}(j).$$
We note that by the log-concavity of $\Tilde{A}_{n,k}$, we have
\begin{align*}
    \Tilde{A}_{n, k}\sum_{i, j = 0}^{\floor{n/2}} \Tilde{A}_{n - i - j, k}\sigma_{-1}(i)\sigma_{-1}(j) \leq \sum_{i, j = 0}^{\floor{n/2}} \Tilde{A}_{n - i, k}\Tilde{A}_{n - j, k}\sigma_{-1}(i)\sigma_{-1}(j) = \bigp{\sum_{i = 0}^{\floor{n/2}} \Tilde{A}_{n - i, k}\sigma_{-1}(i)}^2.
\end{align*}
Thus, it follows that
$$\frac{A_{n, k}A_{n, k + 2}}{A_{n, k + 1}^2} \leq \frac{k}{k + 1}\bigp{1 + O\bigp{(2\log n / 2)^{- k} + \frac{k^2 \log^2 n}{\sqrt{n}}}}.$$
Since we assume that $n \gg k^6\log^6 k$, with the implicit constant sufficiently large, the big-O term is at most $\frac{1}{k}$. In this case, we get $A_{n, k}A_{n, k + 2} \leq A_{n, k + 1}^2$ as desired.
\end{proof}

\section{Proof of Theorem \ref{thm:main-1}(2)}
\subsection{Unsigned Stirling numbers of the first kind}
Let $\left[{n\atop m}\right]$ denote the absolute values of the Stirling numbers of the first kind, i.e. they satisfy $$\sum_m \left[{n\atop m}\right]t^m = t(t+1)\cdots (t+n-1),$$ and let $H_n$ denote the $n$-th harmonic number. Sibuya \cite{Sibuya88} proved the following inequality: $$\frac{\big[{n\atop m}\big]}{\big[{n\atop m-1}\big]}\le \frac{n-m+1}{(n-1)(m-1)}H_{n-1} \leq \frac{H_{n - 1}}{m - 1},$$ which gives us, for $m\ge 2H_n+1$, that
\begin{equation*}\begin{aligned}
\frac{\big[{n+1\atop m+t+1}\big]}{\big[{n+1\atop m+1}\big]} &\le \left(\frac{H_{n}}{m}\right)^t \leq 2^{-t}.
\end{aligned}
\end{equation*}

The following lemma is useful to our proof.
\begin{lemma}\label{lemma:stirling}
Let $r=\ceil{\log_2 n}$, and we are given a sequence $\{k_j\}$. Define $s_j=2\ceil{H_{k_j}}+r+1$ for $k_j\neq 0$ and $0$ otherwise, and take their sum $s=\sum_j s_j$. We have \begin{equation}\label{eq:monotone}\sum_{\substack{ \ l_1+\cdots +l_n=s, \\ l_j\le k_j.}}\prod_j \bigb{{k_j+1 \atop l_j+1}}\le \sum_{\substack{ \ l_1+\cdots +l_n=s-r, \\ l_j\le k_j.}}\prod_j \bigb{{k_j+1 \atop l_j+1}}.\end{equation}
\end{lemma}
\begin{proof}
Let $p$ be the index of the first term satisfying $l_p\ge s_p$. We write $l_j'=l_j$ for $j\neq p$ and $l_p'=l_p-r$. Recall that $$\bigb{{n+1\atop m+t+1}}\le 2^{-t}\bigb{{n+1\atop m+1}},$$ when $m\ge 2H_n+1$. We note that $$l_p'=l_p-r\ge s_p-r=2\ceil{H_{k_p}}+1,$$ and so we obtain \begin{equation}\label{eq:stirlingi}\prod_j \bigb{{k_j+1 \atop l_j+1}}\le 2^{-\ceil{\log_2n}}\prod_j \bigb{{k_j+1 \atop l_j'+1}}\le \frac{1}{n}\prod_j \bigb{{k_j+1 \atop l_j'+1}}.\end{equation}

For each tuple $(l_1,l_2,\ldots,l_n)$ such that $\sum_j l_j=s$ and $l_j\le k_j$, we let $i_0$ be the first $i$ such that $l_i \geq s_i$, and send it to $(l_1,\cdots, l_{i_0 - 1}, l_{i_0} - r, l_{i_0 + 1}, \cdots, l_n)$. Combining \eqref{eq:stirlingi} and this correspondence and noting that each term of the right hand side's summation of \eqref{eq:monotone} has at most $n$ preimages, we obtain our result.
\end{proof}

\begin{proof}[Proof of Theorem \ref{thm:main-1}(2)]
We let $k_j:=\#\{i\mid \lambda_i=j\}$. By Corollary 2 of \cite{HN18}, we have $$\sum_{k = 0}^n A_{n,k} z^k=\sum_{\abs{\lambda} = n}\prod_{j=1}^n\binom{k_j+z}{k_j}.$$ 

Since all the roots of the polynomial \begin{equation}\label{eq:realrooted}\prod_{j=1}^n\binom{k_j+z}{k_j}=\sum_k q_k z^k\end{equation} are real, the coefficients $\{q_k\}$ form a log-concave thus unimodal sequence \cite{Stanley06}. We next prove that the mode is at most $O(\sqrt{n}\log n)$ for these $k_j$ satisfying the obvious identity $k_1+2k_2+\cdots +nk_n=n$.

Using \eqref{eq:realrooted}, we directly calculate $$q_k=C_0\sum_{\substack{\ l_1+\cdots +l_n=k, \\ l_j\le k_j.}}\prod_j \bigb{{k_j+1 \atop l_j+1}},$$ where the constant $C_0=\prod_j \frac{1}{k_j!}$. By Lemma \ref{lemma:stirling}, we can compare that $$q_s\le q_{s-r}.$$ 

Since $\sum jk_j=n$, the sum $s=\sum s_j=\sum_{k_j\neq 0} O(\log n)$ is of the asymptotic $O(\sqrt{n}\log n)$. The unimodality of $\{q_k\}$ implies that $k\gg \sqrt{n}\log n$ exceeds the mode. Therefore, the coefficients $\{A_{n,k}\}$ are monotonically decreasing in $k$ as we desire. 
\end{proof}
\section{On Conjecture \ref{conjecture} and Theorem \ref{thm:main-2}}
\subsection{Proof of Theorem \ref{thm:main-2}}
As we have seen in Section \ref{sec:2}, the main setback in our method is that we are unable to derive a good asymptotic for $c_{n,k}$. Conjecture \ref{conjecture} is based on numerical computation, and its truth represents the obstruction for establishing Heim and Neuhauser's Conjecture for large $n$. In particular, its truth leads to a vastly improved form of Lemma \ref{lemma:pseudo-log-concave}.
\begin{lemma}
\label{lemma:improved-a}
Assume Conjecture \ref{conjecture}. Then for all $k\geq 2$ and $k^{3/2} \leq n \leq C^k$, we have 
\begin{equation}
    A_{n,k} = \bigp{1 + O\bigp{e^{-n^{0.1}}}}\widehat{A}_{n,k}
\end{equation}
where $\widehat{A}_{n,k}$ is a log-concave sequence in $n$.
\end{lemma}
\begin{proof}
Recall that
\begin{equation*}
A_{n,k} = \frac{1}{k!}\sum_{i = 0}^n p(n - i)c_{i,k}.
\end{equation*}
The sequence is thus almost the convolution of two log-concave sequences. It suffices to trim away the terms $i \geq n - 25$. We first recall from the proof of Lemma \ref{lemma:c-asymp} that 
$$c_{n, k} \leq \frac{n^k}{k!}\bigp{\frac{\pi^2}{6}}^k.$$
Since $\{c_{i, k}\}_{i \leq n}$ is log-concave by the assumed \cref{conjecture}, for any $0 \leq l \leq n - k$ we have
$$\frac{c_{n,k}}{c_{n - l, k}} \leq \bigp{\frac{c_{n,k}}{c_{k, k}}}^{\frac{l}{n - k}}.$$
While by \eqref{eq:partition-asymptotic},
$$p(l) \gg e^{0.5 \sqrt{l}}.$$
Taking $l = \ceil{n^{1/3}}$, we conclude that 
$$p(l)c_{n - l, k} \gg e^{n^{0.1}}c_{n,k}.$$
Since $l > 25$ for $n > 25^3$, it follows that
\begin{equation*}
A_{n,k} = \bigp{1 + O\bigp{e^{-n^{0.1}}}}\frac{1}{k!}\sum_{i = 0}^{n - 26} p(n - i)c_{i,k}.
\end{equation*}
Since both $\{c_{i,k}\}_{(i \leq n)}$ and $\{p(n)\}_{(n \geq 26)}$ are log-concave sequences(again see \cite{DP14}), the sequence
$$\widehat{A}_{n,k} = \frac{1}{k!}\sum_{i = 0}^{n - 26} p(n - i)c_{i,k}$$
is the convolution of two log-concave sequences, thus is log-concave by Hoggar's Theorem(\cite{Hoggar}). So we have shown the lemma.
\end{proof}
\begin{proof}[Proof of Theorem \ref{thm:main-2}]
By Theorem \ref{thm:main-1}, it suffices to show, for all sufficiently large $n$ and $\frac{n^{1/6}}{\log n} \ll  k \ll \sqrt{n}\log n$, that
$$A_{n, k + 1}^2 \geq A_{n,k}A_{n, k + 2}.$$
Since $n \gg \log^{-7}C + 1$, for all $k$ in this range we have $2k^{3/2} \leq n \leq C^k$. By the proof of Theorem \ref{thm:main-1}, we get
$$A_{n,k + 1} = \bigp{1 + O\bigp{e^{-0.4\sqrt{n}}}}\frac{1}{k}\sum_{i = 0}^{\floor{n/2}} A_{n - i, k}\sigma_{-1}(i).$$
By Lemma \ref{lemma:improved-a}, we conclude that 
$$A_{n,k + 1} = \bigp{1 + O\bigp{e^{-n^{0.1}}}}\frac{1}{k}\sum_{i = 0}^{\floor{n/2}} \hat{A}_{n - i, k}\sigma_{-1}(i).$$
Similarly, we have
$$A_{n,k} = \bigp{1 + O\bigp{e^{-n^{0.1}}}}\hat{A}_{n,k}$$
and
$$A_{n,k + 2} = \bigp{1 + O\bigp{e^{-n^{0.1}}}}\frac{1}{k(k + 1)}\sum_{i, j = 0}^{\floor{n/2}} \hat{A}_{n - i - j, k}\sigma_{-1}(i)\sigma_{-1}(j).$$
We note that by the log-concavity of $\hat{A}_{n,k}$ for $k^{3/2} \leq n \leq C^k$, we have
\begin{align*}
    \hat{A}_{n, k}\sum_{i, j = 0}^{\floor{n/2}} \hat{A}_{n - i - j, k}\sigma_{-1}(i)\sigma_{-1}(j) \leq  \bigp{\sum_{i = 0}^{\floor{n/2}} \hat{A}_{n - i, k}\sigma_{-1}(i)}^2.
\end{align*}
Thus, we obtain the desired conclusion. Namely, we have that
\begin{align*}
   A_{n, k}A_{n, k + 2} \leq \frac{k}{k + 1}\bigp{1 + O\bigp{e^{-n^{0.1}}}} A^2_{n, k + 1}.\qed\popQED 
\end{align*}
\end{proof}
\subsection{Numerical Evidence for Conjecture \ref{conjecture}}
We are unable to show Conjecture \ref{conjecture}. Numerical evidence does suggest that Conjecture \ref{conjecture} is likely to hold for $C = 2$. Let $n_0(k)$ denote the smallest $n$ such that $c_{n, k}^2 < c_{n - 1, k}c_{n + 1, k}$. The following table shows the value of $n_0(k)$ for $2 \leq k \leq 13$.
\begin{center}
\begin{tabular}{ |c|c|c|c|c|c|c|c|c|c|c|c|c| } 
 \hline
 $k$ & 2 & 3&4&5&6&7&8&9&10&11&12&13\\ 
 \hline
 $n_0(k)$ & 6 & 21&39&73&135&251&475&917&1801&3595&7259&14787\\
 \hline
\end{tabular}
\end{center}
\begin{remark}
We also note that Conjecture \ref{conjecture} seems to generalize to other series whose terms display a similar behavior, such as 
$$f(z) = \frac{z}{1 - z} + \frac{z^2}{2(1 - z^2)}.$$
Investigating this phenomenon might be interesting on its own.
\end{remark}

\end{document}